\documentclass[reqno]{amsart}

\usepackage{a4,amsmath,amssymb,amscd,verbatim,bbm,graphicx}

\newtheorem{theorem}{Theorem}
\newtheorem{proposition}[theorem]{Proposition}
\newtheorem{corollary}[theorem]{Corollary}
\newtheorem{lemma}[theorem]{Lemma}

\theoremstyle{definition}

\newtheorem{remark}[theorem]{Remark}
\newtheorem{remarks}[theorem]{Remarks}

\numberwithin{equation}{section}
\numberwithin{theorem}{section}

\newcommand\numberthis{\addtocounter{equation}{1}\tag{\theequation}}

\begin{document}
\bibliographystyle{plain} \title[Central limit theorem for periodic orbits]% \lebn]
{A central limit theorem for periodic orbits of hyperbolic flows} 

\author{Stephen Cantrell}
\address{Mathematics Institute, University of Warwick,
Coventry CV4 7AL, U.K.}
\email{S.J.Cantrell@warwick.ac.uk}

\author{Richard Sharp} 
\address{Mathematics Institute, University of Warwick,
Coventry CV4 7AL, U.K.}
\email{R.J.Sharp@warwick.ac.uk}

\begin{abstract}
\noindent We consider a counting problem in the setting of hyperbolic dynamics. Let $\phi_t : \Lambda \to \Lambda$ be a weak mixing hyperbolic flow. We count the proportion of prime periodic orbits of $\phi_t$, with length less than $T$, that satisfy an averaging condition related to a H\"older continuous function $f: \Lambda \to \mathbb{R}$. We show, assuming an approximability condition on $\phi$, that as $T \to \infty$, we obtain a central limit theorem. 
\end{abstract}

\maketitle

\section{Introduction}
Let $\phi_t : \Lambda \to \Lambda$ be a hyperbolic flow. By a celebrated result of Ratner \cite{rat},
 a central limit theorem 
holds for H\"older observables with respect to the equilibrium state of a H\"older continuous function
and, in particular, with respect to the measure of maximal entropy $\mu$. More precisely, let
 $f : \Lambda \to \mathbb R$ be a H\"older continuous function
and write
\[
\sigma_f^2 :=\lim_{T \to \infty} \int_\Lambda \left(\int_0^T f(\phi_t(x))dt - T \int f \, d\mu\right)^2 d\mu(x);
\]
Ratner showed that if $\sigma_f^2>0$ then
\[
\mu\left(\left\{x \in \Lambda \hbox{ : } \frac{\int_0^T f(\phi_t x) \, dt - T\int f \, d\mu}{\sqrt{T}} \le y\right\}\right)
\to\frac{1}{\sqrt{2\pi \sigma_f^2}} \int_{-\infty}^y e^{-u^2/2\sigma_f^2} \, du,
\]
as $T \to \infty$.
Furthermore, she showed that $\sigma_f^2>0$ if and only if $f$ is not cohomologous to a constant,
where we say 
 that two functions $f$ and $g$ are cohomologous if $f-g=u'$, with $u : \Lambda \to \mathbb R$ is continuously differentiable along flow lines and
\[
u'(x) := \lim_{t\to 0} \frac{u(\phi_t x) - u(x)}{t}.
\]

In this paper we shall be interested in a periodic orbit version of the above result. 
(We restrict to the case where the flow is weak-mixing. If the flow is not weak-mixing then,
after introducing a symbolic model for the dynamics, we may reduce to the case of a constant suspension flow 
over a subshift of finite type, in which case the 
desired periodic orbit result follows from section 6 of \cite{cp}.)
First let us introduce some terminology. Let $\mathcal P$ denote the set of prime periodic $\phi$-orbits. For $\gamma \in \mathcal P$,
we shall write $l(\gamma)$ for its least period.  We then write
\[\mathcal P(T) = \{\gamma \in \mathcal P \hbox{ : } l(\gamma) \leq T\}\]
and, for $\Delta>0$,
\[
\mathcal P(T,\Delta) = \{\gamma \in \mathcal P \hbox{ : } T < l(\gamma)\le T+\Delta\}.
\]
We also write $\pi(T) = \#\mathcal P(T)$
and
$\pi(T,\Delta) = \#\mathcal P(T,\Delta)$.

For a function $f : \Lambda \to \mathbb R$, we write
\[
l_f(\gamma) = \int_0^{l(\gamma)} f(\phi_t(x)) \, dt,
\]
where $x$ is any point on $\gamma$, and call this the $f$-weight of $\gamma$. 
We say that $f : \Lambda \to \mathbb R$ has {\it integer periods} if 
\[
\{l_f(\gamma) \hbox{ : } \gamma \in \mathcal P\} \subset \mathbb Z
\]
and that $f : \Lambda \to \mathbb R$ is 
%non-lattice
%if $f$ is cohomologous to $a+b\psi$, where $a,b\in \mathbb R$ and $\psi$ is a continuous function 
%with integer periods.
{\it flow independent} if, for $a,b \in \mathbb R$, $a +bf$ has
integer periods only if $a=b=0$.
%(See Remark \ref{context} below.)

The periodic orbits of $\phi_t$ are equidistributed with respect to the measure of maximal
entropy,
in the sense that, for any $\Delta>0$,
\[
\lim_{T \to \infty} \frac{1}{\pi(T,\Delta)} \sum_{\gamma \in \mathcal P(T,\Delta)} \frac{l_f(\gamma)}{l(\gamma)}
= \int f \, d\mu
\]
\cite{bow-equi}, \cite{parry-equi}, and one can formulate a periodic orbit version of the 
central limit theorem to quantify deviations from this equidistribution.
Such a result was first obtained by Lalley
\cite{lalley-87} but it only holds under the assumption that $f$ is 
flow independent, which is strictly stronger than $\sigma_f^2>0$.
(There is also a $C^\infty$ condition on $f$  in Lalley's work but this is easy to remove.) 
Furthermore, Lalley obtained his 
central limit theorem as a consequence of a local limit theorem, the proof of which requires considerable analytic machinery. 
(See Remark \ref{context} below.)
It is therefore interesting to obtain a short and direct proof which holds for all
H\"older continuous $f$ with $\sigma_f^2>0$. This is the purpose of the current paper.
Our proof applies whenever $\phi_t$ is a transitive Anosov flow with 
stable and unstable foliations which are not jointly integrable or, for general hyperbolic flows, whenever
$\phi_t$  satisfies a mild Diophantine condition on the periods of its periodic orbits.
These conditions allow use to apply the work of Dolgopyat \cite{dd}
to give bounds on iterates of a family of so-called transfer operators and hence extensions and bounds on the complex generating functions we need to study.

%Consider a weak mixing hyperbolic flow $\phi_t : \Lambda \to \Lambda$. Let $\mathcal P(T)$ 
%denote the set of prime periodic orbits with least period at most $T$.
%It is well known that the prime orbit counting function $\pi(T) = \#\mathcal P(T)$ satisfies the following 
%asymptotic formula:
%$$ \pi(T) \sim \frac{e^{hT}}{hT}.$$
%Here  $\sim$ denotes that the quotient of the related functions converge to $1$ as $T \to \infty$. 
%This result is known as the Prime Orbit Theorem.

%Much work has been done on calculating the error term in the Prime Orbit Theorem in different 
%settings. For example, the prime orbit counting function for the geodesic flow on a compact surface of 
%variable negative curvature exhibits an exponential error term i.e
%$$ \pi(T) = \mathrm{li}(e^{hT}) + O(e^{cT}),$$
%for some $c<h$ \cite{exp}. The proof of this result due to Pollicott and Sharp \cite{exp} relies on 
%methods of Dolgopyat to bound (the norm of) iterates the transfer operator. These bounds are used to 
%understand the domain of analyticity of an appropriate dynamical zeta-function. We use similar ideas %in this paper.  We require the following assumption so that we can apply the methods shown in %\cite{poly}, which in turn rely on the results of Dolgopyat \cite{dd}. 

Recall that a real number $\beta$ is Diophantine if there exists $c>0$ and $\alpha >1$ such that
$|q\beta - p| \ge c q^{-\alpha}$
for all integers $p,q$ with $q>0$.
We say that 
$\phi_t$ satisfies the {\it approximability condition} 
%if it has two periodic orbits $\gamma_1, \gamma_2 \in \mathcal P$ such that
%$l(\gamma_1)/l(\gamma_2)$ is Diophantine.
if it has three closed orbits $\gamma_1, \gamma_2$ and $\gamma_3$ %with distinct lengths 
such that
\[
 \frac{l(\gamma_1) - l(\gamma_2)}{l(\gamma_2)-l(\gamma_3)}
\]
is Diophantine. 

\medskip

%When this condition holds, the error term in the Prime Orbit Theorem decays polynomially, i.e
%$$ \pi(T) = \frac{e^{hT}}{hT} + O\left(\frac{e^{hT}}{T^{1+\delta}}\right)$$
%for some $\delta >0$ \cite{poly}. However, in general, weak mixing hyperbolic flows can exhibit 
%arbitrarily bad decay rates for the error term in the Prime Orbit Theorem. 

%We consider a counting problem related to periodic orbits for hyperbolic flows. Suppose $\gamma$ is 
%a closed periodic orbit of $\phi$ and that $F: \Lambda \to \mathbb{R}$ a H$\ddot{\text{o}}$lder %continuous function. Let $l(\gamma)$ denote the length of $\gamma$ and write 
%$$l_F(\gamma) = \int_\gamma F.$$  

Our main result is the following. By replacing $f$ with $f - \int f\, d\mu$, it is natural to assume that
$\int f \, d\mu=0$.

\begin{theorem} \label{main}
Suppose that $\phi_t : \Lambda \to \Lambda$ is either a transitive Anosov flow 
with stable and unstable foliations which are not jointly integrable
or a hyperbolic flow satisfying the approximability condition.
Let $f: \Lambda \to \mathbb{R}$ be a H\"older continuous function satisfying
$\int f \, d\mu=0$ that is not a coboundary. Then, for each fixed $\Delta>0$,
%
%\noindent
%(i)
%$$\frac{1}{\pi(T)} \#\left\{\gamma \in \mathcal P(T) : \frac{\int_\gamma f}{ \sqrt{l(\gamma)}} 
%\le y \right\} \to \frac{1}{\sqrt{2\pi} \sigma_f} \int_{-\infty}^y e^{-t^2/2\sigma_f^2} dt,$$
%
%\noindent
%(ii)
\[
\frac{1}{\pi(T,\Delta)} \#\left\{\gamma \in \mathcal P(T,\Delta) : \frac{l_f(\gamma)}{ \sqrt{T}} \le y \right\} \to \frac{1}{\sqrt{2\pi} \sigma_f} \int_{-\infty}^y e^{-t^2/2\sigma_f^2} dt,
\]
for each $y\in\mathbb{R}$, as $T \to \infty$. 
%Here
%$$\sigma_F^2=\lim_{T \to \infty} \int_\Lambda \left(\int_0^T F(\phi_t(x))dt - T \int Fd\mu\right)^2 d
%\mu(x)>0$$
%and $\mu$ is the measure of maximal entropy for $\phi_t$.
\end{theorem}

\begin{remarks} \label{context}

\noindent
(i) The requirement that the flow has non-jointly integrable stable and unstable foliations and 
the approximability condition each imply that the flow is weak-mixing.

\smallskip
\noindent 
(ii)
It is interesting to note that Lalley's periodic orbit version of the local limit theorem for 
hyperbolic flows predates the measure version. In fact, Waddington states a measure local limit 
theorem
for hyperbolic flows in \cite{wadd},
based on the ideas in \cite{lalley-87} and \cite{mul}, although his proof contains some technical gaps.
(In particular, the passage to functions defined on a one-sided shift  on page 459 of
\cite{wadd} needs further justification.)
%(The function $\alpha_{N,v,a}(T)$ defined on page 471 of \cite{wadd} is not monotone and consequently 
%the subsequent application of a Tauberian theorem is not
%justified.)
For semiflows satisfying some abstract conditions (which hold, for example, for a suspension semiflow over 
the map $x \mapsto kx$ mod 1 on the circle $\mathbb R/\mathbb Z$, for an integer $k \geq 2$),
a local limit limit theorem was obtained by Iwata \cite{iwata}.
Dolgopyat and N\'andori recently gave a proof of the local limit theorem  using a different approach
\cite{dol-nan}.

\smallskip
\noindent
(iii)
The approximability condition is not robust under perturbation of the flow. However, Field,
Melbourne and T\"or\"ok \cite{FMT} have given conditions which hold for an open dense set of flows. 
More precisely, if for $r \ge 2$, $\mathcal A_r(M)$ denotes the set of $C^r$ Axiom A flows on a compact manifold
$M$ then $\mathcal A_r(M)$ contains a $C^2$-open, $C^r$-dense subset which satisfies the conditions for every non-trivial basic set.
\end{remarks}

In the next section, we define hyperbolic and Anosov flows and discuss some of their basic properties, 
including the information we will need about entropy and pressure. In section 3,
we mention how our central limit theorem will follow from the pointwise convergence of a family 
of Fourier transforms and introduce a  dynamical $L$-function whose analytic properties will be 
key for our analysis. The work of Dolgopyat \cite{dd} is crucial here.
In section 4, we carry out some calculations using contour integration to obtain an asymptotic formula
for a summatory function related to one we require but containing extra terms. In section 5, we remove these extra terms and complete the proof of Theorem  \ref{main}.

The authors are grateful to Ian Melbourne, Vesselin Petkov and Luchezar Stoyanov for helpful comments.

%Before proving this theorem, we recall some basic facts about hyperbolic flows and the symbolic 
%systems used to code them.

\section{Hyperbolic flows and their periodic orbits}

We begin with the definition of a hyperbolic flow.
Let $\phi_t:M \to M$ be a $C^1$ flow on a smooth manifold $M$ and let $\Lambda \subset M$ be a compact flow invariant subset. We say that $\phi_t : \Lambda \to \Lambda$ is a hyperbolic flow if the following conditions are satisfied.
\begin{enumerate}
\item There is a splitting of the tangent bundle $T_\Lambda M = E^0 \oplus E^s \oplus E^u$ such that
\begin{enumerate}
\item there exist $C,\lambda >0$ with $\|D\phi_t |_{E^s}\|, \|D\phi_{-t}|_{E^u}\| \le Ce^{-\lambda t},$ for $t \ge 0$,
\item $E^0$ is one-dimensional and tangent to the flow.
\end{enumerate}
\item The periodic orbits of $\Lambda$ are dense and $\Lambda$ is not a single orbit.
\item $\Lambda$ contains a dense orbit.
\item There exists an open set $U \supset \Lambda$ such that $\Lambda = \bigcup_{t=-\infty}^\infty \phi_t(U)$.
\end{enumerate}
If (1) holds with $\Lambda =M$ then we say that $\phi_t : M \to M$ is an Anosov flow.
In this case (2) is automatically satisfied (this is the Anosov closing lemma \cite{anosov}) and (4) is trivially satisfied. If, in addition, (3) holds then we say that $\phi_t$ is a transitive Anosov flow.

We say that $\phi_t$ is topologically weak-mixing if it does not admit a non-trivial eigenfrequency corresponding to a continuous function, i.e if the only $G \in C(\Lambda,\mathbb C)$ and $a \in \mathbb R$ such that $G\circ \phi_t = e^{iat} G$ for all $t \in \mathbb{R}$ are the constant functions and $a=0$.
It is known that $\phi_t$ is topologically weak-mixing if and only if
$\{l(\gamma) \hbox{ : } \gamma \in \mathcal P\}$ does not lie in a discrete subgroup of $\mathbb R$.

\begin{proposition} [\cite{poly}]  \label{pot-poly}
Under the hypotheses of Theorem \ref{main}, there exists $\eta>0$ such that
\begin{enumerate}
\item[(i)]
\[
\pi(T) = \frac{e^{hT}}{hT} \left(1+ O\left(\frac{1}{T^\eta}\right)\right),
\]
\item[(ii)]
\[
\sum_{l(\gamma) \leq T} l(\gamma) = \frac{e^{hT}}{h} \left(1+ O\left(\frac{1}{T^\eta}\right)\right).
\]
\end{enumerate}
\end{proposition}

We now recall some of the thermodynamic formalism associated to the flow $\phi_t$.
This is standard material which may be found in, for example, \cite{marsur}.
Let $\mathcal M(\phi)$ denote the set of $\phi_t$-invariant Borel probability measures on $M$.
We define the pressure of a H\"older continuous function $f: M \to \mathbb{R}$ to be
\[
P(f) = \sup_{m \in \mathcal M(\phi)}\left\{h_m(\phi) + \int f \, dm\right\},
\]
where $h_m(\phi)$ denotes the entropy of $\phi_t$ with respect to $m$.
The supremum is attained for a unique measure $m_f \in \mathcal M(\phi)$, which we call
the equilibrium state of $f$. When $F=0$, we call $m_0$ the measure of maximal entropy for 
$\phi_t$ and write $\mu=m_0$.
We have $P(0) = h$, the topological entropy of $\phi$.
For $s \in \mathbb R$, the function $s \mapsto P(sf)$ is real analytic, furthermore 
\[
\frac{dP(sf)}{ds}\Bigg|_{t=0} = \int f \, d\mu
\]
and
\[
\frac{d^2P(sf)}{ds^2}\Bigg|_{t=0} = \sigma_f^2.
\]
Recall that $\sigma_f^2>0$
unless $f$ is cohomologous to a constant.
We may also extend $P(sf)$ to an analytic function for complex values of $s$ in a sufficiently small 
neighbourhood of the real line. 
In particular, $s(t) := P(itf)$ is defined and real analytic for $|t|<\delta$, for some $\delta>0$.

%Suppose that $\int f \, d\mu =0$. Then we have the following lemma.

%For $t$ sufficiently small we have the following expression for the pressure function - 
%see page 267 of \cite{mul}.

The following lemma is a consequence the above discussion.

\begin{lemma}
If $\int f\, d\mu =0$ and $f$ is not a coboundary then, for $|t| < \delta$, 
\[
s(t) = h-\frac{\sigma_f^2t^2}{2} + O(t^3),
\]
with $\sigma_f^2>0$. 
\end{lemma}

A simple calculation then gives
\begin{equation} \label{usefullimit}
\lim_{T \to \infty} e^{(h-s(t/\sqrt{T}))T} = e^{-\sigma_f^2 t^2/2}.
\end{equation}

\section{Fourier Transforms and $L$-functions}

Let $f : \Lambda \to \mathbb R$ be a H\"older continuous function satisfying
$\int f \, d\mu=0$.
By  L\'evy's Continuity Theorem \cite{wiley}, to prove Theorem \ref{main} 
it is enough to to show that the Fourier transforms of the 
distributions
\[
\frac{1}{\pi(T,\Delta)} \#\left\{\gamma \in \mathcal P(T,\Delta) : 
\frac{l_f(\gamma) }{ \sqrt{T}} \le y \right\}
\]
converge pointwise to the Fourier transform of the normal distribution $N(0,\sigma_f^2)$. In other words, we need to show that, for all $t \in \mathbb R$,
\[
\frac{1}{\pi(T,\Delta)} \sum_{\gamma \in \mathcal P(T,\Delta)} e^{itl_f(\gamma)/\sqrt{T}} \to e^{-\sigma_f^2t^2/2},
\]
as $T \to \infty$. To do this, we will use the periodic orbit data $l(\gamma)$ and $l_f(\gamma)$
to build a family of dynamical $L$-functions
$L(s,t)$. Here $s$ is a complex variable (associated to the lengths $l(\gamma)$) and $t$ is a real variable 
(associated to the $f$-weights $l_f(\gamma)$).
%\begin{theorem}
%Let $F_1,F_2,....$ be distributions with Fourier transforms $\gamma_1(t),\gamma_2(t),...$. Suppose that for all 
%$t \in \mathbb{R}$, $\gamma_n(t) \to \gamma(t)$ pointwise as $n \to \infty$ and that $\gamma$ is the Fourier %transform of some distribution $F$. Then, $F_n \rightarrow F$ in distribution.
%\end{theorem}
%To prove Theorem $0.1$ we use the standard method; we study the pointwise convergence of the appropriate 
%sequence of Fourier transforms. Specifically, we show that for fixed $t \in \mathbb{R}$,
%\[
%\frac{1}{\pi(T)} \sum_{l(\gamma) \le T} e^{itl_F(\gamma)/\sqrt{T}} \to e^{-\sigma_F^2t^2/2},
%\]
%as $T \to \infty$. The initial term in the above convergence is the Fourier transform of the distribution
%\[
%\frac{1}{\pi(T)} \#\left\{\gamma \in \mathcal P(T) : \frac{\int_\gamma F }{ \sqrt{T}} \le y \right\}
%\]
%and the second term is the Fourier transform of the Gaussian distribution with mean $0$ and variance 
%$\sigma_F^2$. Theorem $0.1$ then follows from an application of the 
%To prove the above convergence result, we study an appropriate $L$-function.
%\begin{definition}
We define
$$L(s,t) = \prod_{\gamma \in \mathcal P} \left(1-e^{-sl(\gamma) + itl_f(\gamma)}\right)^{-1} = \exp\left\{ \sum_{\gamma \in \mathcal P} \sum_{m=1}^{\infty} \frac{1}{m}e^{-sml(\gamma) + itl_f(\gamma)} \right\}.$$
%\end{definition}
%where the sum and product over $\gamma$ are taken over prime periodic orbits. 
%Let $h$ denote the topological entropy of $\phi$. 
Then $L(s,t)$ is non-zero and analytic in the region 
$\mathrm{Re}(s)>h$ and for all $t \in \mathbb R$ \cite{pp}. 
In order to prove Theorem \ref{main}, we will need to extend $L(s,t)$
to a neighbourhood of $\mathrm{Re}(s)=h$.

In fact, it will be convenient to work with the 
logarithmic derivative (with respect to $s$) $L'(s,t)/L(s,t)$. Write
$\mathcal Q$ for a set of all (not necessarily prime) periodic orbits of $\phi_t$ 
and, for $\gamma' \in \mathcal Q$, if $\gamma' = \gamma^n$, $n \geq 1$, with $\gamma \in \mathcal P$,
write $\Lambda(\gamma') = l(\gamma)$.
Then we have
$$\frac{L'(s,t)}{L(s,t)} = - \sum_{\gamma' \in \mathcal Q} \Lambda(\gamma') e^{-sl(\gamma') + itl_f(\gamma')},$$
whenever the series converges.
%Here $\Lambda(\gamma')$ denotes the prime length of $\gamma'$, i.e if $\gamma'=\gamma^n$, $n \ge 1$,
%with $\gamma \in \mathcal P$ then $\Lambda(\gamma') = l(\gamma)$. 

%For $\delta>0$ sufficiently small, we may define a real analytic function 
%$s: (-\delta,\delta) \to \mathbb{C}$ by 
%$s(t) = P(itF)$. Then
%$s(0) = h$ and $\mathrm{Re}(s) <h$ for $t \neq 0$.

Our proof relies heavily on the following proposition.

\begin{proposition} \label{Lnbhd}
There exists $C>0$ and $\epsilon >0$ 
%and a real analytic function $s: (-\delta,\delta) \to \mathbb{C}$ 
such that, for any fixed 
$t \in (-\delta,\delta)$,
$$\frac{L'(s,t)}{L(s,t)} + \frac{1}{s-s(t)}$$
is analytic in $\mathrm{Re}(s) > h - C \min \{ 1, |\mathrm{Im}(s)|^{-\epsilon}\}.$ 
%for $C, \epsilon>0$ independent of $t \in U$. 
Furthermore, there exists $ \beta>0$, independent of $t \in (-\delta,\delta)$, 
such that for  $\mathrm{Re}(s) > h - C \min \{ 1, |\mathrm{Im}(s)|^{-\epsilon}\},$
$$\left|\frac{L'(s,t)}{L(s,t)}\right| = O\left(\max\{|\mathrm{Im}(s)|^{\beta},1\}\right).$$
\end{proposition}

%The proof of this proposition relies on the symbolic coding (via Lemma \ref{coding} and equation 
%(\ref{bowenmanning})) of our hyperbolic flow. 
%In particular, we will need to pass to a one-sided subshift of finite type $\sigma : \Sigma^+ \to \Sigma^+$,
%where
%\[
%\Sigma^+ = \{(x_n)_{n=0}^{\infty} : x_n \in \{1,2,...,k\}, A_{x_n,x_{n+1}}=1, n \in \mathbb{Z^+}\}.
%\]
%In view of Lemma ??,
%the functions $r: \Sigma \to \mathbb R$ and  $f : \Sigma \to \mathbb R$ defined by 
%\[
% f(x)= \int_0^{r(x)} F\circ \pi(x,u) \ du
%\]
%may be interpreted as H\"older continuous functions $r: \Sigma^+ \to \mathbb R$ 
%and $f : \Sigma^+ \to \mathbb R$.
%We may then consider the symbolic $L$-function
%\[
%L_\sigma(s,t) = \exp \sum_{n=1}^\infty \frac{1}{n} \sum_{\sigma^n(x)=x} e^{-sr^n(x) + itf^n(x)}.
%\]
%It is well known that $L_\sigma(s,t)$ is non-zero and analytic for $\mathrm{Re}(s)>h$ (for any value of 
%$t \in \mathbb R$) and, due to equation \ref{bowenmanning},
%$L(s,t)/L_\sigma(s,t)$ is non-zero and analytic for $\mathrm{Re}(s) > h-\epsilon$, for some $\epsilon>0$
%(independent of $t$).
%Thus it will be sufficient to prove Proposition \ref{Lnbhd} for the function $L_\sigma(s,t)$.

\begin{proof}
%Our aim is to understand the analyticity of $L$ as discussed above. To do so we analyse the symbolic version 
%of the $L$-function
%\[
%\tilde{L}(s,t) = \exp\left\{\sum_{n=1}^\infty \frac{1}{n} \sum_{\sigma^n(x)=x} e^{-sr^n(x) + itf^n(x)} \right\}.
%\]
%Here we have abused notation and have replaced $r$ and $f$ by H$\ddot{\text{o}}$lder continuous functions 
%from the one-sided shift space
%\[
%\Sigma^+ = \{(x_n)_{n=0}^{\infty} : x_n \in \{1,2,...,k\}, A_{x_n,x_{n+1}}=1, n \in \mathbb{Z}_{\ge 0}\}
%\]
%to $\mathbb{R}$. These new functions agree with the old ones on the set of periodic points and 
%thus do not change the $L$-function. There is a standard method of comparing the properties of $\tilde{L}$ 
%to those of $L$. In fact, to prove this proposition it suffices to show the analogous statement for $\tilde{L}$. 
%We admit these details to simplify the following exposition - See \cite{man}, \cite{poly}, Lemma $3.3$ and %Lemma $3.6$ of \cite{han}.
It is a standard part of the theory of dynamical zeta functions that $L(s,t)$ 
has a simple pole 
at $s=s(t)$ and, apart from this pole, is analytic and non-zero for $s$ close to $s(t)$ \cite{pp}.
Hence, $L'(s,t)/L(s,t)$ has a simple pole at $s=s(t)$ with residue $-1$.

The extension to a larger domain and the bound rely on the work of Dolgopyat on bounds for iterates of transfer
operators \cite{dd}.
In the case where $t=0$, the extension and bound were established \cite{poly}.
For $t$ non-zero but small, one may modify the approach in \cite{dd} to get the required results.
(See for example \cite{PetStoy} where similar calculations are carried out.
The recent paper \cite{Stoy}, as well as containing important new results, gives a detailed account of the history of this problem.)
%
%Using the method shown in \cite{poly} which relies on the work on Dolgopyat \cite{dd}, we deduce that there 
%exists a neighbourhood of the origin $U,$ an analytic function $s: U \to \mathbb{R}$ and $C,\epsilon >0$ 
%independent of $t \in U$ such that $\tilde{L}(s,t)$ is analytic and non-zero for fixed $t \in U$ in
%$$W = \left\{s: \mathrm{Re}(s)>h-c\min\{|\mathrm{Im}(s)|^{-\epsilon},1\} \text{ and } s \neq s(t)\right\}.$$
%This method requires that $\phi$ has the approximability condition. This domain of analyticity passes to the 
%logarithmic derivative. It is standard (see \cite{pp} or \cite{han}) to check that for fixed $t \in U$, 
%$\frac{d}{ds}(
%\log\tilde{L})(s,t)$ has a simple pole of residue $-1$ at $s=s(t)$. Hence we deduce the first part of the  
%proposition.
%
%For the second part of the proposition we use the method presented in Lemma $3$ and Lemma $4$ of 
%\cite{poly}. By Lemma $3$ in \cite{poly}, thinning the region $W$ and reducing $U$ if necessary, we have that, %for $s \in W$,
%$$\log\left(\tilde{L}(s,t)\right) = O(\min\{|\mathrm{Im}(s)|^p\log|\mathrm{Im}(s)|,1\}),$$
%for some $p$ independent of $t \in U$. 
%
%This can be converted into a bound for the logarithmic derivative using Lemma $4$ in \cite{poly}.
\end{proof}

%The function $s: U \to \mathbb{R}$ has an interpretation in terms of thermodynamic formalism; 
%$s(t)$ denotes the pressure of the function $itF.$ Hence, $s(0) = h$. The following lemma provides 
%insight into the behavior of $s$ for small real $t$. Let $\text{P}$ denote the analytic extension of the 
%pressure function that agrees with $s$ on $U$.

%\begin{lemma}
%For real non-zero $t \in U$, $\mathrm{Re}(s(t)) = \mathrm{Re}(\text{P}(itF)) < h$ and 
%hence for fixed non-zero real $t \in U$,
%$$x^{s(t)} = O(x^h).$$
%\end{lemma}

%\begin{proof}
%This follows easily from the inequality 
%$$\mathrm{Re}(\text{P}(f)) \le \text{P}(\mathrm{Re}(f))$$ 
%(which holds for any H$\ddot{\text{o}}$lder function $f: \Sigma \to \mathbb{C}$) and the convexity of the 
%pressure function.
%\end{proof}

\section{Contour Integration}

%From this point onwards, we assume that $t$ is real valued. To avoid overcomplicating the following 
%exposition, we will abuse notation and write $U$ to denote $U \cap \mathbb{R}$.

The rest of the proof follows similar lines to the method used in section $2$ of \cite{hol}. We need the following 
standard identity (see \cite{ingham}, page 31), which holds for any $k\ge 1$,
\begin{equation} \label{perron}
\frac{1}{2\pi i} \int_{d-i\infty}^{d+i\infty} \frac{x^{s}}{s(s+1)\dotsm(s+k)} ds = \left\{
     \begin{array}{@{}l@{\thinspace}l}
       0 &\hspace{4pt} 0< x < 1\\
       \frac{1}{k!} (1-1/x)^k &\hspace{4pt} x \ge 1. \\
      
     \end{array}
   \right.
 \end{equation}
 %\vspace{2mm}
 
Applying (\ref{perron}) term-by-term to $-L'(s,t)/L(s,t)$ gives
\[
\sum_{e^{l(\gamma')} \le x} 
\Lambda(\gamma') e^{itl_f(\gamma')} (x - e^{l(\gamma')})^k 
= \frac{k!}{2\pi i} \int_{d-i\infty}^{d+i\infty} 
\left(-\frac{L'(s,t)}{L(s,t)}\right)\frac{x^{s+k}}{s(s+1)\dotsm (s+k)} ds,
\]
where $\gamma'$ runs over the elements of $\mathcal Q$.

\begin{lemma} \label{contour}
For any fixed $t \in (-\delta,\delta)$, there is $k \ge 1$ and $\alpha>0$ such that
\[
\sum_{e^{l(\gamma')} \le x} \Lambda(\gamma') e^{itl_f(\gamma')} (x - e^{l(\gamma')})^k = \frac{k!}{s(t)(s(t)+1) \dotsm (s(t)+k)} x^{s(t)+k} + O\left(\frac{x^{h+k}}{(\log x)^\alpha}\right).
\]
The implied constant in the above error term is independent of $t$.
\end{lemma}

\begin{proof}
We first note that for $s=\varsigma +i\tau$, with $\varsigma >h$, we have the trivial estimate
\begin{equation} \label{trivial}
\left|\frac{L'(s,t)}{L(s,t)}\right| \leq \left|\frac{L'(\varsigma,0)}{L(\varsigma,0)}\right|
=O\left(\frac{1}{\varsigma-h}\right).
\end{equation}

We will prove the result using contour integration.
Choose numbers $r_1,r_2>0$ such that, for $|t|<\delta$, 
$s(t)$ lies in the interior of the rectangle with vertices at $r_1 + ir_2$, $h+ir_2$, $h-ir_2$ and $r_1 - ir_2$. where 
$r_1, r_2 >0$. Furthermore, decreasing $\delta$ if necessary, we may 
assume that this rectangle lies within the region of analyticity described in Proposition \ref{Lnbhd}. 
Set $c= h - CR^{-\epsilon}/2$, $d=h+(\log x)^{-1}$ and $R=(\log x)^\kappa$ where $0<\kappa<1/\epsilon$. Take $k > \beta$, where $\beta$ is the same as in Proposition \ref{Lnbhd}. By the Residue Theorem we may write
\begin{align*}
\sum_{e^{l(\gamma')} \le x} \Lambda(\gamma') e^{itl_f(\gamma')} (x - e^{l(\gamma')})^k &= \frac{k!}{s(t)(s(t)+1) \dotsm (s(t)+k)} x^{s(t)+k} +\\
&\hspace{6mm} \frac{k!}{2 \pi i} \int_\Gamma \left(\frac{L'(s,t)}{L(s,t)}\right) \frac{x^{s+k}}{s(s+1) \dotsm (s+k)} ds,
\end{align*}
where $\Gamma$ is the contour consisting of the straight lines connecting the points $d+i\infty,$ $d+iR$, $c+iR$, $c+ir_2$, $r_1+ir_2$, $r_1-ir_2$, $c-ir_2$, $c-iR$, $d-iR$ and $d-i\infty$.

%\centerline{\includegraphics[width= 4cm, height=7cm]{Picture}}
%\begin{center}
%Figure $1$\\
%\end{center}
Using (\ref{trivial}), we have
\[
\left|\int_{d \pm iR}^{d \pm i\infty} \frac{L'(s,t)}{L(s,t)} \frac{x^{s+k}}{s(s+1)\dotsm (s+k)} \, ds \right| = O\left(x^{d+k} \int_R^\infty \frac{1}{u^{k+1}} \, du \right) = O\left(\frac{x^{h+k}}{R^{k}}\right).
\]
Using the bound from Proposition \ref{Lnbhd}, we have
\begin{align*}
%&\left|\int_{d \pm iR}^{d \pm i\infty} \frac{L'(s,t)}{L(s,t)} \frac{x^{s+k}}{s(s+1)\dotsm (s+k)} \, ds \right| = 
%O\left(x^{d+k} \int_R^\infty \frac{1}{u^{k+1}} \, du \right) = O\left(\frac{x^{h+k}}{R^{k}}\right),\\
&\left| \int_{c \pm iR}^{d \pm i R} \frac{L'(s,t)}{L(s,t)} \frac{x^{s+k}}{s(s+1) \dotsm (s+k)} \, ds \right| = O\left(R^{\beta - k- 1} x^{c+k}\right),\\
&\left| \int_{c \pm ir_2}^{c\pm iR} \frac{L'(s,t)}{L(s,t)} \frac{x^{s+k}}{s(s+1)\dotsm(s+k)} \, ds \right| = O\left(x^{c+k}\right),\\
&\left| \int_{c \pm ir_2}^{r_1\pm ir_2} \frac{L'(s,t)}{L(s,t)} \frac{x^{s+k}}{s(s+1)\dotsm(s+k)} \, ds \right| = O\left(x^{c+k}\right),\\
&\left| \int_{r_1 + ir_2}^{r_1 -  ir_2} \frac{L'(s,t)}{L(s,t)} \frac{x^{s+k}}{s(s+1)\dotsm(s+k)} \, ds \right| = O\left(x^{r_1+k}\right).
\end{align*}
From our choice of $c$, $d$, $\kappa$, $R$ and $k$, we see that
$O(x^{h+k}/R^k)$ and $O(R^{\beta-k-1} x^{c+k})$ are $O(x^{h+k}/(\log x)^\alpha)$, for some 
$\alpha>0$, while
$O(x^{c+k})$ is  $O(x^{h+k}a(x))$, where $a(x)$ tends to zero faster than $(\log x)^{-\eta}$, for any 
$\eta>0$. The final term has a power saving compared to $x^{h+k}$.
Thus the result follows.
%We deduce that the claim follows for $0<\alpha < \min\{ (k-\beta)/\epsilon, 2\delta'\}$.
\end{proof}

We claim that the previous lemma holds if we alter the sum so that it is taken over prime orbits.
 We define
\[
S_k(x) =\sum_{e^{l(\gamma)} \le x} \Lambda(\gamma) e^{itl_f(\gamma)} (x - e^{l(\gamma)})^k,
\]
where the summation is over $\gamma \in \mathcal P$.

\begin{corollary} \label{cor-contour}
For any fixed $t \in (-\delta,\delta)$, there is $k \ge 1$ and $\alpha>0$ such that
\[
%\sum_{e^{l(\gamma)} \le x} \Lambda(\gamma) e^{itl_F(\gamma)} \left(x - e^{l(\gamma)}\right)^k 
S_k(x)
= \frac{k!}{s(t)(s(t)+1) \dotsm (s(t)+k)} x^{s(t)+k} + O\left(\frac{x^{h+k}}{(\log x)^\alpha}\right).
\]
The implied constant in the above error term is independent of $t$.
\end{corollary}

\begin{proof}
Let $l_0$ be the shortest length of any periodic orbit for $\phi$. We write
\begin{align*}
\sum_{e^{l(\gamma')} \le x} \Lambda(\gamma') e^{itl_F(\gamma')} (x-e^{l(\gamma')})^k = %\sum_{e^{l(\gamma)} \le x} l(\gamma) &e^{itl_F(\gamma)} \left(x - e^{l(\gamma)}\right)^k\\
S_k(x)
+ \sum_{n=2}^\infty \sum_{e^{nl(\gamma)}\le x} l(\gamma) e^{itnl_f(\gamma)}(x-e^{nl(\gamma)})^k.
\end{align*}
Note that in the last sum over $n$ in the above expression, 
the terms are non-zero only for $n \le (\log x)/l_0$. Hence,
$$\sum_{n=2}^\infty \sum_{e^{nl(\gamma)}\le x} l(\gamma) e^{itnl_f(\gamma)}(x-e^{nl(\gamma)})^k = O(\log x \cdot x^{h/2} \cdot \log x \cdot x^k) = O( x^{k+h/2} (\log x)^2 ).$$
This implies the claim.
\end{proof}

\section{Auxiliary Calculations}
We now wish to remove the terms $(x - e^{l(\gamma)})^k$ from $S_k(x)$. 
We will first show that the estimate for $S_k(x)$ in Corollary \ref{cor-contour}
implies a similar estimate for $S_{k-1}(x)$, though with the exponent of $\log x$ in 
the error term reduced.

Decreasing $\alpha$ if necessary, we suppose that $0<\alpha < 2\eta$,
where $\eta$ is as in Proposition \ref{pot-poly}.
Set $\epsilon =(\log x)^{-\alpha/2}$.
We will estimate the difference
\[
D(x,\epsilon):= 
S_k(x(1+\epsilon)) -S_k(x)
%\sum_{e^{l(\gamma)} \le x(1+\epsilon)} l(\gamma) e^{itl_F(\gamma)} (x(1+\epsilon) - e^{l(\gamma)})^k 
%- \sum_{e^{l(\gamma)} \le x} l(\gamma) e^{itl_F(\gamma)} (x - e^{l(\gamma)})^k
 \]
 in two ways.
Applying Corollary \ref{cor-contour}, we have
 \begin{align*}
D(x,\epsilon)
&=
\frac{k!}{s(t)(s(t)+1) \dotsm (s(t)+k)} x^{s(t)+k} \epsilon + O\left(x^{s(t)+k}\epsilon^2\right) + O\left(\frac{x^{h+k}}
{(\log x)^\alpha}\right) \\
&= 
\frac{k!}{s(t)(s(t)+1) \dotsm (s(t)+k)} x^{s(t)+k} \epsilon  + O\left(\frac{x^{h+k}}
{(\log x)^\alpha}\right).
\numberthis \label{leadingterm}
 \end{align*}
 On the other hand,
we have
 \begin{align*} %\label{binomial}
 D(x,\epsilon) &=
 \sum_{x \le e^{l(\gamma)} \le x(1+\epsilon)} l(\gamma) e^{itl_f(\gamma)} ( x - e^{l(\gamma)})^k + kx\epsilon \sum_{e^{l(\gamma)} \le x} l(\gamma) e^{itl_f(\gamma)} (x- e^{l(\gamma)})^{k-1}\\
&+ \sum_{j=2}^k(x\epsilon)^j { k \choose j} \sum_{e^{l(\gamma)} \le x} l(\gamma) e^{itl_f(\gamma)} (x - e^{l(\gamma)})^{k-j}. \numberthis \label{binomial}
\end{align*}
Rewriting part (ii) of Proposition \ref{pot-poly}, we have
\[
\sum_{e^{l(\gamma)} \le x} l(\gamma) \sim \frac{x^h}{h}
\]
and, since $\alpha/2 < \eta$,
\[
\sum_{x \le e^{l(\gamma)} \le x(1+\epsilon)} l(\gamma) \sim \epsilon x^h.
\]
Hence, the first term on the Right Hand Side of (\ref{binomial})
is $O(x^{h+k}\epsilon^k)$ and the $j$th term in the final summation is
$O(x^{h+k} \epsilon^j)$. Dividing by $kx\epsilon$ and comparing with (\ref{leadingterm})
gives
\begin{align*}
%\sum_{e^{l(\gamma)} \le x} l(\gamma) e^{itl_F(\gamma)} (x-e^{l(\gamma)})^{k-1} 
S_{k-1}(x)
&=\frac{(k-1)!}{s(t)(s(t)+1) \dotsm (s(t)+k-1)} x^{s(t)+k-1}+O\left(\frac{x^{h+k-1}}{\epsilon (\log x)^\alpha}, \epsilon x^{h+k-1}\right) \\
&=\frac{(k-1)!}{s(t)(s(t)+1) \dotsm (s(t)+k-1)} x^{s(t)+k-1}
+O\left(\frac{x^{h+k-1}}{(\log x)^{\alpha/2}}\right).
\end{align*}
Proceeding inductively, we obtain the following (where the new value of $\alpha$ is the original
$\alpha$ divided by $2^k$.)

\begin{lemma} \label{kremoved}
For some $\alpha>0$, we have 
$$ \sum_{e^{l(\gamma)} \le x} l(\gamma)e^{itl_F(\gamma)} = \frac{x^{s(t)}}{s(t)} + 
O\left(\frac{x^{h}}{(\log x)^{\alpha}}\right).$$
\end{lemma}

We note that the constant associated to the above error term is independent of $t \in (-\delta,\delta)$.

\section{Proof of Theorem \ref{main}}

We will now complete the proof of Theorem \ref{main}.
A simple calculation using Lemma \ref{kremoved} and the limit (\ref{usefullimit}) gives the following.

\begin{lemma}
For any $t \in \mathbb{R}$,
$$ \sum_{l(\gamma) \le T} l(\gamma)e^{it l_f(\gamma)/\sqrt{T}} \sim \frac{e^{s(t/\sqrt{T})T}}{s(t/\sqrt{T})}.$$
\end{lemma}

%Setting $F=0$ we deduce
%\begin{corollary}
%$$ \sum_{l(\gamma) \le T}l(\gamma) \sim \frac{e^{hT}}{h}.$$
%\end{corollary}

This lemma, together with part (ii) of Proposition \ref{pot-poly} and (\ref{usefullimit})
again, implies that
\[
\lim_{T \to \infty} \frac{ \sum_{l(\gamma)\le T} l(\gamma) 
e^{itl_f(\gamma)/\sqrt{T}}}{\sum_{l(\gamma) \le T} l(\gamma)} 
\to e^{-\sigma_f^2t^2/2},
\]
provided we now assume that $\sigma_f^2>0$, i.e. that $f$ is not a coboundary.

We now need to remove the terms $l(\gamma)$.
 From Proposition \ref{pot-poly}, we have that
\begin{equation} \label{removel}
\sum_{l(\gamma) \le T} l(\gamma) \sim T \pi(T),
\end{equation}
as $T \to \infty$.
We also have the following.

\begin{lemma} \label{stieltjes}
For any $t \in \mathbb R$,
$$\sum_{l(\gamma) \le T} e^{itl_f(\gamma)/\sqrt{l(\gamma)}} \sim \frac{1}{T} \sum_{l(\gamma)\le T} l(\gamma) e^{itl_f(\gamma)/\sqrt{T}}.$$
\end{lemma}

\begin{proof}
Let $\varphi(T) = \sum_{l(\gamma) \le T} l(\gamma) e^{itl_f(\gamma)/\sqrt{T}}$. Using the Stiltjes integral, we have that

\begin{align*}
\sum_{l(\gamma) \le T} e^{itl_f(\gamma)/\sqrt{l(\gamma)}} &= \int_1^T \frac{1}{u} \, d \varphi(u) + O(1)\\
&=\left[ \frac{\varphi(u)}{u}\right]_1^T + \int_1^T  \frac{\varphi(u)}{u^2} \, du + O(1)\\
&=\frac{\varphi(T)}{T} + O(1) + O\left(\int_1^T \frac{e^{hu}}{u^2} \, du \right).
\end{align*}
Integration by parts yields the estimate
$$\int_1^T \frac{e^{hu}}{u^2} \ du = \left[\frac{e^{hu}}{hu^2}\right]_1^T + 2 \int_1^T\frac{e^{hu}}{u^3} \ du = O\left(\frac{e^{hT}}{T^2} \right)$$
and the result follows.
\end{proof}

Combining Lemma \ref{stieltjes} and (\ref{removel}) gives us the following.

\begin{proposition} \label{nearlythere}
$$\lim_{T \to \infty}\frac{1}{\#\pi(T)} \sum_{l(\gamma) \le T} e^{itl_f(\gamma)/\sqrt{l(\gamma)}} = \lim_{T \to \infty} \frac{ \sum_{l(\gamma)\le T} l(\gamma) e^{itl_f(\gamma)/\sqrt{T}}}{\sum_{l(\gamma) \le T} l(\gamma)} = e^{-\sigma_f^2t^2/2}.$$ 
\end{proposition}

We now complete the proof of Theorem \ref{main}.

%Not needed:
%\[
%\pi(T,T+1) \sim (e^h-1)\pi(T)
%\]

\begin{proof}[Proof of Theorem \ref{main}]
First note that it follows from the prevision calculations (and 
$\pi(T,\Delta) = \pi(T+\Delta)-\pi(T)$) that
\[
\lim_{T\to \infty} \frac{1}{\pi(T,\Delta)} \sum_{T<l(\gamma)\le T+\Delta} e^{itl_f(\gamma)/\sqrt{l(\gamma)}}
=  e^{-\sigma_f^2 t^2/2}.
\]
Then note that, for a fixed $t$,
\[
\sum_{T<l(\gamma)\le T+\Delta} \left|e^{itl_f(\gamma)/\sqrt{l(\gamma)}}
-
 e^{itl_f(\gamma)/\sqrt{T}}\right|
=O\left(\frac{\pi(T,\Delta)}{\sqrt{T}}\right).
\]
This gives use the required convergence,
\[
\lim_{T\to \infty} \frac{1}{\pi(T,\Delta)} \sum_{T<l(\gamma)\le T+\Delta} e^{itl_f(\gamma)/\sqrt{T}}
=  e^{-\sigma_f^2 t^2/2}.
\]

\end{proof}

\begin{remark}
In view of Proposition \ref{nearlythere}, Theorem \ref{main} may be reformulated as
%Suppose that $\phi_t : \Lambda \to \Lambda$ is either a weak-mixing transitive Anosov flow 
%or a weak mixing hyperbolic flow satisfying the approximability condition.
%Let $f: \Lambda \to \mathbb{R}$ be a H\"older continuous function satisfying
%$\int f \, d\mu=0$ that is not a coboundary. Then
$$\frac{1}{\pi(T)} \#\left\{\gamma \in \mathcal P(T) : \frac{l_f(\gamma)}{ \sqrt{l(\gamma)}} \le y \right\} \to \frac{1}{\sqrt{2\pi} \sigma_f} \int_{-\infty}^y e^{-t^2/2\sigma_f^2} dt,$$
for each $y\in\mathbb{R}$, as $T \to \infty$. 
\end{remark}

\end{document}